\documentclass[10pt,leqno]{amsart}
\usepackage{graphicx}
\usepackage{indentfirst,csquotes}

\usepackage{amssymb,amsthm,amsmath}
\usepackage{xcolor}
\usepackage{paralist}

\usepackage{fancyhdr}
\usepackage{etoolbox}
\newcommand\commentone[1]{\textcolor{black}{#1}}
\newcommand\commenttwo[1]{\textcolor{black}{#1}}
\newtheorem{theorem}{Theorem}[section]

\newtheorem{Lemma}[theorem]{Lemma}
\newtheorem{Proposition}[theorem]{Proposition}
\newtheorem{Corollary}[theorem]{Corollary}

\newtheorem{Remark}[theorem]{Remark}

\usepackage[ruled,vlined]{algorithm2e}
\usepackage{lipsum}
\usepackage{hyperref}
\begin{document}
\title{A Monotone Discretization for the Fractional Obstacle Problem} 
\author[Rubing Han, Shuonan Wu, Hao Zhou]{Rubing Han, Shuonan Wu, Hao Zhou}
\date{\today}
\address{School of Mathematical Sciences, Peking University, Beijing, 100871, People's Republic of China.}
\email{hanrubing@pku.edu.cn, snwu@math.pku.edu.cn, 2301110061@pku.edu.cn}
\maketitle

\let\thefootnote\relax
\footnotetext{{\bf Funding:} This research was funded by the National Natural Science Foundation of China grant No. 12222101 and the Beijing Natural Science Foundation No. 1232007.} 

\begin{abstract}
We introduce a novel monotone discretization method for addressing obstacle problems involving the integral fractional Laplacian with homogeneous Dirichlet boundary conditions over bounded Lipschitz domains. 
\commenttwo{This problem is prevalent in mathematical finance, particle systems, and elastic theory.} 
By leveraging insights from the successful monotone discretization of the fractional Laplacian, we establish uniform boundedness, solution existence, and uniqueness for the numerical solutions of the fractional obstacle problem. 
We employ a policy iteration approach for efficient solution of discrete nonlinear problems and prove its finite convergence. 
Our improved policy iteration, adapted to solution regularity, demonstrates superior performance by modifying discretization across different regions. 
Numerical examples underscore the method's efficacy.
\end{abstract} 

\smallskip
\noindent \textbf{Keywords:} obstacle problem; fractional Laplacian; bounded Lipschitz domain; monotone discretization; policy iteration.

\smallskip

\noindent \textbf{MSC:} 35R11, 65N06, 65N12, 65N15.


\section{Introduction}


In this work, we consider the obstacle problem associated with the
integral form of the fractional Laplacian, referred to as the {\it
fractional obstacle problem}. 

For $n\geq 1$, let $\Omega\subset \mathbb{R}^n$ be a bounded Lipschitz
domain satisfying the exterior ball condition, with boundary denoted
as $\partial \Omega$. Additionally, $\Omega^c$ represents the
complement of $\Omega$. The specific form of the fractional obstacle
problem is as follows: Given two prescribed functions $f:\Omega
\rightarrow \mathbb{R}$ and the obstacle function $\psi : \mathbb{R}^n
\rightarrow \mathbb{R}$ with the nondegeneracy condition
$\psi\vert_{\Omega^c} < 0$, and $s\in(0,1)$, we seek a function
$u:\mathbb{R}^n\rightarrow \mathbb{R}$ that satisfies the following
nonlinear equation with homogeneous Dirichlet boundary conditions:
\begin{equation}\label{eq:fractional_obstacle}
    \left\{
\begin{aligned}
  \mathcal{G}[u] :=\min \left\{(-\Delta )^s u - f, u - \psi \right\}
  &=0,\; x\in\Omega,\\ u &= 0 ,\; x\in \Omega^c.
\end{aligned}
\right.
\end{equation}
Here, the integral fractional Laplacian of order $s\in (0,1)$, defined
by 
\begin{equation} \label{eq:fractional_laplace}
    (-\Delta)^s u(x) := C_{n,s} \text{P.V.} \int_{\mathbb{R}^n}
    \frac{u(x) - u(y)}{|x-y|^{n+2s}}\mathrm{d}y.
\end{equation}
The normalization constant is given by $C_{n,s}:= \frac{2^{2s}
s\Gamma(s+\frac{n}{2})}{\pi^{n/2} \Gamma(1-s)}$.

The fractional obstacle problem finds applications in various fields,
such as mathematical finance theory, systems of particles and elasticity problems
\cite{tankov2003financial,carrillo2016regularity,serfaty2018systems,signorini1933sopra,fernandez2022thin}. For a comprehensive
overview of obstacle problems, including traditional obstacle problems
and those related to integral-differential operators (of which the
fractional obstacle problem is a particular case), we refer to the
survey article \cite{ros2018obstacle}.



\commenttwo{For fractional operators, it has been shown that the maximum principle holds \cite{ros2016nonlocal}. 
In terms of numerical stability, it is desirable to preserve this property at the discrete level, known as the monotonicity of the scheme. 
Constructing monotone schemes is a natural advantage of finite difference methods.
For the Laplace operator $(-\Delta)$, the monotone difference scheme has been paid attention to by researchers for a long time, \cite{bramble1962formulation,van1974high}. In the context of the finite element method, achieving monotonicity in the discretization process depends on specific grid conditions \cite{xu1999monotone}.}

\commenttwo{Furthermore, considering the regularity of the solution is essential for the convergence of numerical schemes.}
The fractional obstacle problem has received significant attention in mathematics, particularly in the field of PDE regularity of solutions
and the free boundary. In this regard, a series of works
\cite{silvestre2007regularity,caffarelli2007extension,caffarelli2008regularity}
have established the H\"older regularity of $C^{1,s}(\mathbb{R}^n)$
for solutions to the $s$-order obstacle problem. For the case of a
bounded domain, it is expected that the interior regularity exhibits
the same $C^{1,s}$ H\"older regularity. However, lower regularity may
arise near the boundary of the domain. For the fractional Laplace
equation, a well-known result is that when the problem domain
satisfies the exterior ball condition and the right-hand side belongs
to $L^\infty$, the solution exhibits singularity near the boundary
$\partial \Omega$ in terms of $\mathrm{dist}(x,\partial \Omega)^s$,
resulting in global $C^s$ H\"older regularity. This characteristic is
observed widely in a large class of nonlocal elliptic equations, as
discussed in the survey article \cite{ros2016nonlocal}. Therefore, the
global regularity (boundary regularity) of the fractional obstacle
problem on bounded domains still requires careful investigation.

In recent years, numerous numerical algorithms have been developed for
solving fractional-order partial differential equations (PDEs). For
the fractional Laplace equation, the finite element method can be
employed, utilizing the variational formulation and accounting for the
low regularity near the boundary \cite{acosta2017fractional}.
Similarly, the fractional obstacle problem can be formulated as a
variational inequality problem, where the solution $u$ is sought to
minimize a functional subject to the constraint $u \geq \psi$
\cite{silvestre2007regularity}. To address this variational
inequality, \cite{borthagaray2019weighted} utilizes the finite element
method, establishing interior and boundary regularity results for
bounded fractional obstacle problems, and providing error estimates
based on these regularity results.

However, for
fractional operators, due to their non-local nature, it is challenging
to establish general conditions that guarantee the monotonicity of
finite element discretization. This challenge is further amplified
when dealing with the finite element discretization of nonlinear
problems.

\commenttwo{On the other hand, several finite difference discretizations that satisfy
monotonicity have been proposed
\cite{huang2014numerical,duo2018novel,duo2019accurate} for the fractional operator
$(-\Delta)^s$.} However, most of these works
focus on problems solved on structured grids and make overly strong
assumptions about the H\"older regularity of the true solution. In our
recent work \cite{han2022monotone}, we propose a monotone difference
scheme based on quadrature formulas, which can be applied to solve the
fractional Laplace equation on general bounded Lipschitz domains. This
work provides a comprehensive analysis of the scheme's consistency,
taking into account the actual regularity of the problem. Furthermore,
we obtain rigorous pointwise error estimates by utilizing discrete
barrier functions.

The objective of this study is to develop a monotone scheme for
solving the problem defined by equation \eqref{eq:fractional_obstacle}
on a bounded Lipschitz domain and to devise an efficient solution
solver. Building upon our previous work \cite{han2022monotone}, we
extend this discretization to the fractional obstacle problem. A
crucial aspect of our work is the introduction of the enhanced
discrete comparison principle, as demonstrated in Lemma
\ref{lm:enhanced-DCP}, for the discrete fractional Laplacian operator.
By utilizing this result, we establish that the discretization scheme
maintains monotonicity and satisfies the discrete maximum principle,
guaranteeing the uniqueness of the discrete solution. Additionally, we
employ the discrete Perron method to establish the existence and
uniform boundedness of the discrete solution.

In addition, this study explores the policy iteration method for
solving nonlinear discrete problems. Leveraging the enhanced discrete
comparison principle, we establish the finite convergence of the
standard policy iteration method. Furthermore, the paper provides an
in-depth discussion on the relationship between the discretization
scheme for the fractional Laplace equation and the regularity results
of the problem. Based on this discussion, we propose an improved
policy iteration method that considers the low regularity near the
contact set of the problem and selects different discrete scales for
different regions. Numerical results demonstrate the superior
performance of the improved method compared to the standard policy
iteration method.

The remaining part of the paper is organized as follows. In Section
\ref{sec:prelim}, we introduce some necessary preliminary results. In
Section \ref{sec:discretization}, we present the monotone
discretization for the fractional obstacle problem and discuss the
properties of the scheme.  In Section \ref{sec:solver}, we apply the
policy iteration method for solving nonlinear discrete problems and
prove the convergence of the iteration. Additionally, taking into
account the specific regularity of the problem, we propose an improved
policy iteration method that exhibits better numerical performance in
practical computations.  Numerical examples are presented in Section
\ref{sec:numerical} to support our theoretical results. Furthermore,
for the boundary H\"older regularity of the problem, a simple proof is
provided in Appendix \ref{ap:holder}.

\section{Preliminary results} \label{sec:prelim}
In this section, some preliminary knowledge and results will be introduced, including notation, definitions of function spaces, and regularity results.
Given an open set $U\subset \mathbb{R}^n$ with $\partial U \neq \varnothing$, the $C^\beta(U)$ H\"older seminorm, with $\beta > 0$, is denoted by $|\cdot |_{C^\beta(U)}$. More precisely, for $\beta = k +\commentone{t}$ with $k$ integer and $\commentone{t} \in (0,1]$, define
\[
\begin{aligned}
	|w|_{C^\beta(U)}  = |w|_{C^{k,\commentone{t}}(U)}&: = \sup\limits_{x,y\in U, x\neq y} \frac{|D^k w(x) - D^kw(y)|}{|x-y|^{\commentone{t}}},\\
	\|w\|_{C^\beta(U)} &: = \sum_{l=0}^{k} \left(\sup\limits_{x\in U} |D^l w(x)| \right) + |w|_{C^\beta(U)}.
\end{aligned}
\]

As usual, the nonessential constant denoted by $C$ may vary from line to line. $X\lesssim Y$ means that there exists a constant $C>0$ such that $X\leq CY$, and $X\eqsim Y$ means that $X\lesssim Y$ and $Y \lesssim X$. 

\subsection{Regularity: the linear problem}
Here are some well-known conclusions regarding the regularity of the integral fractional Laplacian:
\begin{equation}\label{eq:Deltas}
	(-\Delta)^s w_f = f \text{ in } \Omega,\qquad w_f = 0 \; \text{in } \Omega^c. 
\end{equation}

The first lemma pertains to the interior regularity of fractional harmonic functions. This result establishes that the solution of \eqref{eq:fractional_obstacle} is smooth within the region unaffected by $\psi$, or in other words, where there is no contact with $\psi$.
\begin{Lemma}[balayage] \label{lm:balayage}
	Let $w\in L^\infty(\mathbb{R}^n)$ be such that $(-\Delta)^s w = 0$ in $B_R$. Then, $w\in C^\infty(B_{R/2})$.
\end{Lemma}

For the global regularity result, however, it can be proven that the solution belongs to $C^s(\mathbb{R}^n)$, and this result is sharp \cite{ros2014dirichlet}. This is due to the limited regularity at the boundary and suggests the need for a more comprehensive discussion using weighted Hölder spaces, as discussed in \cite{ros2014dirichlet, acosta2017fractional}.

\subsection{Regularity: the fractional obstacle problem} \label{subsec:regularity}
Next, some regularity results for the fractional obstacle problem are presented. It is worth noting that, when discussing regularity, the assumption that the forcing term $f$ is zero can be made without loss of generality. This is because the general solution can be decomposed as follows: 
\[
u = u_1 + w_f,
\]
where $u_1$ solves \eqref{eq:fractional_obstacle} with zero forcing ($f = 0$) and obstacle $\psi_1 := \psi - w_f$. Here, $w_f$ is the solution to the linear problem \eqref{eq:Deltas}.

The contact set is defined as follows:
\[
\Lambda = \{x\in \Omega: u(x) = \psi(x)\}.
\]
The first result on the interior H\"older regularity of the fractional obstacle problem was established by Silvester in $\mathbb{R}^n$ \cite{silvestre2007regularity}. When considering the fractional obstacle problem in bounded domains, Nochetto, Borthagaray, and Salgado in \cite{borthagaray2019weighted} derived a corresponding result under the assumption:
\begin{equation} \label{eq:ass-contact-dist}
\text{dist}(\Lambda,\partial \Omega) = r_0 > 0.
\end{equation}
The main technique employed in their work is the use of truncation functions. The result is presented as follows.

\begin{Proposition}[interior H\"older regularity] \label{pp:interior-Holder}
    Let $\Omega$ be a bounded Lipschitz domain and $\psi\in C(\bar{\Omega}) \cap C^{2,1}(\Omega)$ with nondegeneracy condition $\psi|_{\Omega^c} < 0$. Then the solution $u$ of \eqref{eq:fractional_obstacle}, with $f=0$, satisfies $u\in C^{1,s}(D)\;\text{\commentone{ for all }}$ \commentone{$ D\subset \Omega$ is compact}.
\end{Proposition}

Similarly, utilizing truncation function techniques, \cite{borthagaray2019weighted} provided the following boundary H\"older regularity result:
\begin{Proposition}[boundary H\"older regularity] \label{pp:boundary-Holder}
    Let $\Omega$ be a bounded Lipschitz domain satisfying the exterior ball condition, and $\psi\in C(\bar{\Omega}) \cap C^{2,1}(\Omega)$ with nondegeneracy condition $\psi|_{\Omega^c} < 0$. Let $u$ solves \eqref{eq:fractional_obstacle}, then
    \[
    \|u\|_{C^{s}(\mathbb{R}^n)}\leq C(\Omega, f,\psi).
    \]
\end{Proposition}

In fact, in the boundary Hölder regularity estimate, the smoothness assumption on $\psi$ can be relaxed to $\psi \in L^\infty(\Omega)$. The main idea is to revert back to the fundamental theoretical framework of boundary regularity \cite{ros2014dirichlet} and explore its extension to the fractional obstacle problem. Additionally, compared to \cite{borthagaray2019weighted}, the proof is more direct. The detailed proof of this new approach is included in the Appendix \ref{ap:holder}. 

Before concluding this section, let us provide an intuitive summary of the implications of the regularity results for the design of numerical schemes: 
\begin{enumerate}
\item Lemma \ref{lm:balayage} (balayage) indicates that the solutions are smooth within both $\Lambda$ and $\Omega \setminus \Lambda$ if $\psi$ is smooth. Therefore, the standard discretization methods would be suitable for these regions.
\item By combining the interior regularity discussed in Proposition \ref{pp:interior-Holder}, it suggests that the smoothness across the contact boundary decreases to $C^{1,s}$. As a result, discretization methods that rely on higher-order derivatives would be inappropriate for this region.
\item The boundary regularity result, Proposition \ref{pp:boundary-Holder}, is similar to that of linear problems. Thus the techniques used for linear problems, such as the regularity under weighted norms, are still applicable. From a numerical perspective, this phenomenon motivates the use of graded grids to enhance the convergence order, as explored in the work by \cite{acosta2017fractional, han2022monotone}. 
This issue will be revisited in Section \ref{sec:numerical}.
\end{enumerate}

\section{A monotone discretization} \label{sec:discretization}
For simplicity, it is assumed that $\Omega$ is a polygon in 2D and a polyhedron in 3D. Let $\mathcal{T}_h$ be a triangulation of the
computation domain $\Omega$, i.e. $\cup_{T\in \mathcal{T}_h}
\overline{T}  = \overline{\Omega}$. For $T\in \mathcal{T}_h$, let
$h_T$ denote the diameter of element $T$, $\rho_T$ denote the radius
of the largest {inscribed} ball contained in $T$. The triangulation is referred to as \emph{local quasi-uniform} if there exists a constant $\lambda_1$ such that
\[
h_T \leq \lambda_1 h_{T'} \quad \forall \overline{T} \cap
\overline{T'} \neq \varnothing.
\]
Meanwhile, the triangulation $\mathcal{T}_h$ is called \emph{shape
regular}, if there is $\lambda_2 >0 $ such that $h_T \leq \lambda_2
\rho_T$.  For $n\geq 2$, the shape regular property implies the local
quasi-uniform properties, see \cite{brenner2008mathematical}. Let
$\mathcal{N}_h$ denote the node set of $\mathcal{T}_h$ and
$\mathcal{N}_h^b :=\left\{x_i \in \mathcal{N}_h\; : \; x_i\in \partial
\Omega\right\}$ be the collection of boundary nodes, and interior node
set is denoted by $\mathcal{N}_h^0:=\mathcal{N}_h \setminus
\mathcal{N}_h^b$.

Let $\mathbb{V}_h := \left\{ u\in C(\mathbb{R}^n): \; u\vert_T \in
\mathcal{P}_1(T), u\vert_{\Omega^c} = 0 \right\}$, where
$\mathcal{P}_1(T)$ denotes the linear function collection defined on
element $T$.  In \cite{han2022monotone}, a monotone discretization for the integral fractional Laplacian was proposed and a corresponding pointwise error estimate was conducted under realistic H\"older regularity assumption. In this work, the monotone discretization for the integral fractional Laplacian is applied, denoted by $(-\Delta)^s_h$:  

\begin{equation}\label{eq:Delta-hs}
  (-\Delta)_h^s [v_h](x_i) : = \underbrace{-\kappa_{n,s,i}
  \frac{\Delta_{FD} v_h(x_i;H_i)}{H_i^{2s}}}_{:=\mathcal{L}^S_h
  [v_h](x_i)} +\underbrace{ \int_{\Omega_i^c} \frac{v_h(x_i) -
  v_h(y)}{|x_i-y|^{n+2s}}\mathrm{d}y}_{:=\mathcal{L}_h^T[v_h](x_i)}\quad
  \forall v_h\in\mathbb{V}_h.
\end{equation}
Here, $\Omega_i \subset \Omega$ is a star-shaped domain centered at
$x_i$ satisfying some symmetrical condition \cite[Equ. (3.5) and
(3.6)]{han2022monotone} with a typical scaling 
\begin{equation} \label{eq:Hi}
H_i := h_i^{\alpha_i} \delta_i^{1-\alpha_i} \quad \forall x_i \in \mathcal{N}_h^0,
\end{equation}
where $h_i$ is the mesh size around $x_i$, and $\delta_i: =
\text{dist}(x_i,\partial \Omega)$. A concrete example that satisfies
the condition is the $n$-cubic domain $\Omega_i = x_i + [-H_i,
H_i]^n$, which is also utilized in our numerical experiments.  The
parameter $\alpha_i$ is carefully selected to address the regularity
concern of integral fractional Laplacian \cite{ros2014dirichlet,
han2022monotone}.  

The singular integral within $\Omega_i$ is approximated using the
finite difference method, denoted as $\Delta_{FD} u(x_i; H_i) :=
\sum_{j=1}^{n} u(x_i + H_ie_j) - 2u(x_i) + u(x_i- H_ie_j)$, where
$e_j$ represents the unit vector of the $j$th coordinate. The
coefficient $\kappa_{n,s,i}$ is a known positive constant, provided in
\cite[Equ. (3.10)]{han2022monotone}.

\subsection{Properties of $(-\Delta)^s_h$}
In this subsection, several fundamental properties of the $(-\Delta)^s_h$ operator are revisited and established. These properties will be utilized in subsequent discussions. One of the key features is its monotonicity,
as proven in \cite[Lemma 3.4]{han2022monotone}, which is closely related to the matrix discretization being an $M$-matrix. An even stronger property of $(-\Delta)^s_h$ will be explored: the discretization matrix exhibits a strong diagonal dominance. Additionally, the consistency of $(-\Delta)^s_h$ will be discussed.

To begin with, let us revisit the discrete barrier function and
monotone property given in \cite[Lemma 3.4 \& Lemma
3.5]{han2022monotone}. 

\begin{Lemma}[discrete barrier function]\label{lm:discrete-barrier}
Let $b_h\in \mathbb{V}_h$ satisfy $b_h(x_i):=1$ for all $x_i \in \mathcal{N}_h^0$. It follows that
\begin{equation} \label{eq:discrete-barrier}
    (-\Delta)^s_h[b_h](x_i)\geq C\delta_i^{-2s}>0\quad \forall x_i\in
    \mathcal{N}_h^0,
\end{equation}
where the constant $C$ depends only on $s$ and $\Omega$.
\end{Lemma}

\begin{Lemma}[monotonicity of $(-\Delta)^s_h$]\label{lm:monotone-Delta-h}
Let $v_h, w_h\in \mathbb{V}_h$. If $v_h - w_h$ attains a non-negative
  maximum at an interior node $x_i \in\mathcal{N}_h^0$, then
  $(-\Delta)^s_h[v_h](x_i)\geq (-\Delta)^s_h[w_h](x_i)$.
\end{Lemma}

The above monotonicity result is insufficient for the obstacle
problem, as the fractional order operator equation in the obstacle
problem holds only in certain regions of the domain. Let us denote the
resulting discrete matrix of $(-\Delta)_h^s$ as ${\bf L} \in
\mathbb{R}^{N\times N}$, where $N$ is the number of interior nodes.
Below, it will be shown that ${\bf L}$ is a strongly diagonal dominant $M$-matrix.

\begin{Lemma}[strongly diagonal dominant $M$-matrix]
  \label{lm:L-property} 
  The discrete matrix ${\bf L} \in \mathbb{R}^{N \times N}$ satisfies
\begin{subequations} \label{eq:L-property}
\begin{align}
  &{\bf L}_{ii} > 0 \quad \forall i; \qquad {\bf L}_{ij} \leq 0 \quad
  \forall i,j \text{ and } i\neq j; \label{eq:L-property1}\\
  & {\bf L}_{ii} > \sum_{i\neq j}^N|{\bf L}_{ij}| \quad \forall i.
  \label{eq:L-property2}
\end{align}
\end{subequations}
\end{Lemma}
\begin{proof}
The property \eqref{eq:L-property1} follows directly from the
  definition of $(-\Delta)^s_h$. In fact, for the equation
  \eqref{eq:Delta-hs} at $x_i$, it is straightforward to observe that
  ${\bf L}_{ii} = 2n\kappa_{n,s,i} H_i^{-2s} + \int_{\Omega_i^c}
  |x_i-y|^{-(n+2s)}\mathrm{d}y$. Furthermore, all the other
  contributions to the off-diagonal terms are non-positive. 

Utilizing Lemma \ref{lm:discrete-barrier}, the discrete barrier function can be expressed as the vector $B\in \mathbb{R}^N$, where all entries are equal to one. Therefore, \eqref{eq:discrete-barrier} implies that 
$$
({\bf L} B)_i = {\bf L}_{ii} + \sum_{i\neq j}^N{\bf L}_{ij} = {\bf
  L}_{ii} - \sum_{i\neq j}^N|{\bf L}_{ij}|\geq C\delta_i^{-2s} > 0,
$$
which leads to \eqref{eq:L-property2}.
\end{proof}

Now, the enhanced discrete comparison principle applicable to the obstacle problem will be presented. This principle plays a crucial role in the subsequent analysis.

\begin{Lemma}[enhanced discrete comparison principle for
  $(-\Delta)^s_h$]\label{lm:enhanced-DCP}
Let $v_h, w_h \in \mathbb{V}_h$ be such that
    \begin{subequations} \label{eq:DCP}
    \begin{align}
        (-\Delta)^s_h [v_h](x_i) &\geq (-\Delta)^s_h[w_h](x_i)\quad
        \forall x_i \in \Lambda_h, \label{eq:DCP1}\\
        v_h(x_i)&\geq w_h(x_i) \quad \forall x_i \in
        \mathcal{N}_h^0\setminus \Lambda_h, \label{eq:DCP2}
    \end{align}
    \end{subequations}
where $\Lambda_h\subset \mathcal{N}_h^0$ is an arbitrary subset. Then,
  $v_h \geq w_h$ in $\Omega$.
\end{Lemma}
\begin{proof}
Since $v_h, w_h$ are piecewise linear functions, it suffices to prove
$v_h(x_i)\geq w_h(x_i)$ for all $x_i\in\mathcal{N}_h^0$.  Let $z_h :=
v_h - w_h$, and $Z \in \mathbb{R}^N$ denotes its coefficient under the
basis function.  Our objective is to demonstrate that $Z \geq 0$.
Note that the resulting discrete matrix of $(-\Delta)_h^s$, denoted as
${\bf L}$, exhibits strong diagonal dominance. Then, \eqref{eq:DCP}
has the matrix representation
    \[
    {\bf L} Z = 
    \left(
    \begin{matrix}
        {\bf L}_{11} & {\bf L}_{12}\\
        0 & {\bf I} \\
    \end{matrix} 
    \right)
    \left(
    \begin{matrix}
        Z^1\\
        Z^2
    \end{matrix}
    \right)
    \geq 0.
    \]
Here, the indices of all interior points in $\Lambda_h$ are combined into the first group. The remaining indices form the second group. This implies that ${\bf L}{11}Z^{1} + {\bf L}{12} Z^{2} \geq 0$ and $Z^2\geq 0$.
  
By above lemma (strongly diagonal dominant $M$-matrix), the entries in
  the off-diagonal block ${\bf L}_{12}$ are non-positive, which yields
\[
    {\bf L}_{11} Z^{1} \geq -{\bf L}_{12} Z^2 \geq 0.
\]
Since ${\bf L}$ is a strongly diagonal dominant $M$-matrix, so do
${\bf L}_{11}$, that is $({\bf L}_{11}^{-1})_{ij}\geq 0$. Therefore, it can be deduced that $Z^1 \geq 0$ and $Z\geq 0$, which completes the proof.
\end{proof}

\subsection{Numerical scheme} \label{subsec:scheme}
Now, the numerical scheme for solving the fractional obstacle problem is proposed: Find $u_h\in \mathbb{V}_h$, such that
\begin{equation}\label{eq:Gh}
	\begin{aligned}
    \mathcal{G}_h[u_h](x_i)  := \min \big\{(-\Delta)^s_h[u_h](x_i)
    -f(x_i), u_h(x_i)-\psi(x_i) \big\} = 0\quad \forall x_i \in
    \mathcal{N}_h^0.
	\end{aligned}
\end{equation}
It is noted that, although not explicitly stated, there is a parameter $\alpha$ in the discretization of the integral fractional Laplacian that is influenced by the regularity of the solution $u$. The discrete comparison principle for $\mathcal{G}_h$ will be demonstrated, from which the uniqueness of \eqref{eq:Gh} follows directly.

\begin{Lemma}[discrete comparison principle for $\mathcal{G}_h$]\label{lm:DCP-Gh}
Let $v_h, w_h \in \mathbb{V}_h$ be such that
\[
\mathcal{G}_h [v_h] (x_i) \geq \mathcal{G}_h [w_h](x_i)\quad \forall
  x_i \in \mathcal{N}_h^0.
\] 
Then, $v_h \geq w_h$ in $\Omega$.
\end{Lemma}

\begin{proof}
Since $v_h ,w_h\in \mathbb{V}_h$, it suffices to prove $v_h(x_i) \geq
  w_h(x_i)$ for all $x_i \in \mathcal{N}_h^0$. The
  interior points will be classified into two sets based on the values of
  $\mathcal{G}_h[w_h]$:
$$
\Lambda_h^w := \{x_i \in \mathcal{N}_h^0 :  (-\Delta)_h^s [w_h] (x_i)
  - f(x_i) \geq w_h(x_i) - \phi(x_i)\}.
$$ 

For any $x_i \in { \Lambda_h^w}$, the following holds:
$$ 
v_h(x_i) - \psi(x_i) \geq \mathcal{G}_h[v_h](x_i) \geq
  \mathcal{G}_h[w_h](x_i) = w_h(x_i) - \psi(x_i),
$$
Furthermore, for any $x_i \in {\mathcal{N}_h^0 \setminus \Lambda_h^w}$, the following inequalities are true: 
$$
(-\Delta)^s_h[v_h](x_i) - f(x_i) \geq \mathcal{G}_h[v_h](x_i) \geq
  \mathcal{G}_h[w_h](x_i) = (-\Delta)^s_h[w_h](x_i) - f(x_i).
$$ 
This leads to the desired result by taking $\Lambda_h = \Lambda_h^w$
  in Lemma \ref{lm:enhanced-DCP} (enhanced discrete comparison
  principle for $(-\Delta)^s_h$).
\end{proof}

\begin{Corollary}[uniqueness] \label{co:uniqueness}
    For the problem \eqref{eq:Gh}, the discrete solution is unique.
\end{Corollary}

\subsection{Stability and existence}

The existence and stability of \eqref{eq:Gh} will be shown in this section. 

\begin{theorem}[existence and stability] \label{tm:existence} 
  There exists a unique $u_h \in \mathbb{V}_h$ that solves
  \eqref{eq:Gh}. The solution $u_h$ is stable in the sense that
  $\|u_h\|_{L^\infty(\Omega)} \leq C$, where the constant $C$ is
  independent of $h$.
\end{theorem}

\begin{proof} To establish the existence, a monotone sequence of discrete functions $\{\{u_h^k\}_{k=0}^\infty\}$ is constructed, starting with the initial guess $u_h^0 \in \mathbb{V}_h$ that satisfies the following condition: 
\begin{equation} \label{eq:uh0}
\mathcal{G}_h[u_h^0](x_i) \geq 0 \quad \forall x_i \in
  \mathcal{N}_h^0.
\end{equation} 

\noindent {\it Step 1 (existence of $u_h^0$)}. Let $u_h^0 :=
  E b_h \in \mathbb{V}_h$ where the discrete barrier function $b_h$ is
  defined in \eqref{eq:discrete-barrier}, and the constant $E > 0$
  will be specified later. By the Lemma \ref{lm:discrete-barrier}
  (discrete barrier function), the following relation holds:
\[
 \begin{aligned}
        (-\Delta)^s_h[Eb_h](x_i)-f(x_i)&\geq CE
        \delta_i^{-2s}-f(x_i)\geq CE \mathrm{diam}(\Omega)^{-2s} -
        \|f\|_{L^\infty(\Omega)}\\
        Eb_h(x_i) -\psi(x_i) &\geq  E - \|\psi\|_{L^\infty(\Omega)}.
    \end{aligned}
\]
Setting $E = \max\{\|\psi\|_{L^\infty(\Omega)}, C^{-1}
  \mathrm{diam}(\Omega)^{2s}\|f\|_{L^\infty(\Omega)}\}$, then
  \eqref{eq:uh0} is ensured.

\noindent {\it Step 2 (Perron construction)}. Induction is employed.
Suppose that there already exists a discrete function $u_h^k \in
\mathbb{V}_h$ that satisfies
\begin{equation} \label{eq:uhk}
\mathcal{G}_h[u_h^k](x_i) \geq 0 \quad \forall x_i \in
  \mathcal{N}_h^0.
\end{equation} 
The construction of $u_h^{k+1} \in \mathbb{V}_h$ with the properties $u_h^{k+1}
\leq u_h^k$ and satisfaction of \eqref{eq:uhk} is as follows: All interior nodes are considered sequentially, and auxiliary functions $u_h^{k,i-1} \in \mathbb{V}_h$ are constructed using the first $i-1$ nodes, starting with $u_h^{k,0}
:= u_h^k$.
At $x_i \in \mathcal{N}_h^0$, it is checked
whether or not $\mathcal{G}_hu_h^{k,i-1} > 0$. If this is the case, the value of $u_h^{k,i-1}(x_i)$ is decreased, resulting in the function $u_h^{k,i}$, until
$$ 
\mathcal{G}_h[u_h^{k,i}](x_i) = 0.
$$ 
This is achievable since the discrete matrix of $(-\Delta)_h^s$
  satisfies ${\bf L}_{ii} > 0$, as stated in \eqref{eq:L-property1}.
  Moreover, at other nodes $x_j \neq x_i$, ${\bf L}_{ij} \leq 0$
  implies that 
$$
\mathcal{G}_h[u_h^{k,i}](x_j) \geq  \mathcal{G}_h[u_h^{k,i-1}](x_j)
  \geq 0 \quad \forall x_j \neq x_i.
$$
This process is repeated with the remaining nodes $x_j$ for $i < j \leq
N$, and $u_h^{k+1} := u_h^{k,N}$ is set as the last intermediate function. By construction, the following is obtained:
$$
\mathcal{G}_h[u_h^{k+1}](x_i) \geq 0, \quad u_h^{k+1}(x_i) \geq
  u_h^k(x_i) \quad \forall x_i \in \mathcal{N}_h^0.
$$

\noindent {\it Step 3 (convegence)}. The sequence
$\{u_h^k(x_i)\}_{k=1}^\infty$ is monotonically decreasing and clearly
clearly bounded from below by $\psi(x_i)$. Hence, the sequence
converges, and the limit 
$$ 
u_h(x_i) := \lim_{k \to \infty} u_h^k(x_i) \quad \forall x_i \in \mathcal{N}_h^0
$$ 
defines $u_h \in \mathbb{V}_h$. It also satisfies the desired equality 
$$ 
\mathcal{G}_h[u_h](x_i) = 0 \quad \forall x_i \in \mathcal{N}_h^0,
$$ 
since $\mathcal{G}_h[u_h](x_i)  = \lim_{k\to \infty}
\mathcal{G}_h[u_h](x_i) \geq 0$ and if the last inequality were
strict, Step 2 could be applied to further improve $u_h$. This
demonstrates the existence of a discrete solution to \eqref{eq:Gh}.
The above proof also implies
$$ 
\|u_h\|_{L^\infty(\Omega)} \leq \max\{\|\psi\|_{L^\infty(\Omega)},
C^{-1} \mathrm{diam}(\Omega)^{2s}\|f\|_{L^\infty(\Omega)}\},
$$ 
which is the uniform bound, as asserted. 
\end{proof}

\begin{Remark}[discussion on the consistency and convergence]
An important concept in numerical methods is its consistency, which involves investigating the property of 
  $|\mathcal{G}[v](x) - \mathcal{G}_h[\mathcal{I}_h v](x^h_i)|$
  tending to zero as the mesh size decreases. In fact, for the
  integral fractional Laplacian operator, a detailed analysis in
  \cite{han2022monotone} reveals that the discretization achieves
  consistency at interior points that are a constant number of mesh
  sizes away from the boundary. The inconsistency near boundary points
  can be addressed by incorporating a discrete barrier function
  \eqref{eq:discrete-barrier}. This analysis technique is commonly
  used in the convergence analysis of semi-Lagrangian or two-scale method
  \cite{feng2017convergent,nochetto2019two} and can be viewed as an extension to the
  Barles-Souganidis \cite{barles1991convergence} analysis framework
  for nonlinear problems. Due to space limitations, elaboration on the corresponding results will not be provided here.
\end{Remark}

\section{Solver} \label{sec:solver}

In this section, the {\it policy iteration} (Howard's algorithm) is employed to solve the discrete fractional obstacle problem \eqref{eq:Gh}. It is worth noting that policy iteration is a well-established and extensively studied technique in dynamic control
problems. 
For the problem \eqref{eq:Gh}, the convergence of the policy iteration is established by leveraging the monotonicity of the discrete operator. 
Furthermore, an improved policy iteration will be proposed by incorporating prior knowledge, such as the regularity of the solution to the fractional obstacle problem discussed in Section \ref{sec:prelim}.

\subsection{Policy iteration} \label{subsec:policy}

Policy iteration is utilized to solve the discrete problem \eqref{eq:Gh}. The algorithm is outlined as follows.

\begin{algorithm}[htbp!] 
    \caption{Policy iteration for the fractional obstacle problem \eqref{eq:Gh}} \label{al:policy}
    \KwIn{
    right hand side $f$, obstacle function $\psi$.}
    \KwOut{solution $u_h$.}
    \BlankLine
    Initialize $u_h^{(0)}(x_i) = \psi(x_i) \quad \forall x_i \in \mathcal{N}_h^0$;\\
    Iterate for $k \geq 0$:
    \begin{enumerate}
        \item Update discrete contact set $\mathcal{C}^{(k+1)}_h$:\\
        $\mathcal{C}^{(k+1)}_h:=\{x_i: (-\Delta)_h^{s}[u_h^{(k)}](x_i)
        - f(x_i) \geq u_h^{(k)}(x_i) -\psi(x_i) \}$.
        \item Update solution $u_h^{(k+1)}$ such that 
        \begin{equation} \label{eq:policy-update}
        \begin{aligned}
            u_h^{(k+1)}(x_i) -\psi(x_i) &= 0\quad \forall x_i \in
            \mathcal{C}^{(k+1)}_h, \\
            (-\Delta)^{s}_h[u_h^{(k+1)}](x_i) -f(x_i) &=0 \quad
            \forall x_i \in \mathcal{N}_h^0\setminus
            \mathcal{C}^{(k+1)}_h.
        \end{aligned}
        \end{equation}
        \item If $u_h^{(k+1)}=u_h^{(k)}$, then stop; Otherwise,
          continue to Step 1.
    \end{enumerate}
\end{algorithm}

There exists a convergence result for policy iteration when solving
the nonlinear problem \eqref{eq:Gh}, which was initially proven in
\cite{bokanowski2009some}. Here, a brief overview of the proof using our notation is provided for clarity. It is worth noting that the
monotonicity property plays a crucial role in the convergence analysis.

\begin{theorem}[convergence of policy iteration] \label{tm:policy-convergence}
  The sequence $\{u_h^{(k)}\}_{k=0}^\infty$ given by above algorithm
  satisfies:
	\begin{enumerate}
    \item $u_h^{(k)}(x_i) \leq u_h^{(k+1)}(x_i)$ for all $k \geq 0$
      and $x_i \in \mathcal{N}_h^0$;
		\item The algorithm converges in at most $N$ iterations.
	\end{enumerate}
\end{theorem}
\begin{proof}
With $u_h^{(0)}(x_i) = \psi(x_i)$ for all $x_i \in \mathcal{N}_h^0$,
$\mathcal{C}_h^{(0)}$ can be defined as $\mathcal{N}_h^0$ to ensure
that \eqref{eq:policy-update} holds for the initialization.

\noindent{\it Step 1 (increasing sequence)}. \eqref{eq:policy-update} implies that for $k\geq 0$
$$ 
\mathcal{G}_h[u_h^{(k)}](x_i) = \min\{ (-\Delta)_h^{s}[u_h^{(k)}](x_i)
  - f(x_i), u_h^{(k)}(x_i) -\psi(x_i)\} \leq 0 \quad \forall x_i \in
  \mathcal{N}_h^0.
$$ 
Then, the definition of $\mathcal{C}_h^{(k+1)}$ in Step 1 implies 
    \[
    \begin{aligned}
        u_h^{(k)}(x_i) - \psi(x_i)&\leq 0\quad \forall x_i \in
        \mathcal{C}^{(k+1)}_h, \\
        (-\Delta)^{s}_h [u_h^{(k)}](x_i) - f(x_i)&\leq 0\quad \forall
        x_i \in \mathcal{N}_h^0\setminus \mathcal{C}^{(k+1)}_h.
    \end{aligned}
    \]

    Next, using the update condition \eqref{eq:policy-update}, the following relation holds:
    \[
    \begin{aligned}
        u_h^{(k)}(x_i) - \psi(x_i)&\leq  u_h^{(k+1)}(x_i) -
        \psi(x_i)\qquad \quad \forall x_i \in \mathcal{C}^{(k+1)}_h,
        \\
        (-\Delta)^{s}_h [u_h^{(k)}](x_i) - f(x_i)&\leq
        (-\Delta)^{s}_h [u_h^{(k)}](x_i) - f(x_i) \quad \forall x_i
        \in \mathcal{N}_h^0\setminus \mathcal{C}^{(k+1)}_h.
    \end{aligned}
    \]
By combining Lemma \ref{lm:enhanced-DCP} (enhanced discrete
comparison principle for $(-\Delta)_h^s$), the increasing sequence is shown as
    \[
    u_h^{(k)}(x_i)\leq u_h^{(k+1)}(x_i)\quad \forall x_i \in
    \mathcal{N}_h^0.
    \]
    
\noindent{\it Step 2 (strictly decreasing discrete contact set)}. First, the increasing sequence $\{u_h^{(k)}\}_{k=0}^\infty$
  implies  
  \[
      u_h^{(k)}(x_i) - \psi(x_i) \geq 0,\quad \forall x_i \in
      \mathcal{N}_h, \;\forall k\geq 0.
  \]

  Now, for any $x_j\in \mathcal{N}_h^0 \setminus \mathcal{C}^{(k)}_h$,
  the update condition \eqref{eq:policy-update} and the above property
  implies 
  $$
        (-\Delta)_h^{s}[u_h^{(k)}](x_j) - f(x_j) = 0 \leq u_h^{(k)}(x_j)
        - \psi(x_j),
  $$
  whence $x_j \in \mathcal{N}_h^0 \setminus \mathcal{C}^{(k+1)}_h$ for
  the definition of $\mathcal{C}^{(k+1)}_h$ in Step 1. This means 
  $$
  \mathcal{N}_h^0 \setminus \mathcal{C}^{(k)}_h \subset
  \mathcal{N}_h^0 \setminus C^{(k+1)}_h, \quad \text{or} \quad
  \mathcal{C}^{(k+1)}_h \subset \mathcal{C}^{(k)}_h \quad \forall k
  \geq 0.
  $$
  Further, if $\mathcal{C}^{(k+1)}_h = \mathcal{C}^{(k)}_h$, then the
  iteration clearly stops at $(k+1)$-th step, due to the update
  condition \eqref{eq:policy-update}. That is, the discrete contact
  set is strictly decreasing unless the iteration stops, which makes
  at most $N$ iterations. 
\end{proof}

\begin{Remark}[convergence profile of Algorithm \ref{al:policy}]
The $N$-step iteration is proven to be sharp in the general case, as
  demonstrated by a specific example presented in
  \cite{santos2004convergence}. However, in numerical experiments,
  convergence is often observed within just a few iterations.
\end{Remark}

\subsection{An improved policy iteration}
Next, improvements will be introduced to the discrete operator in the policy iteration algorithm. It is observed that the scaling $H_i$ of the singular integral part of $(-\Delta)_h^s$ over the region $\Omega_i$
can be adjusted, as shown in expression \eqref{eq:Hi}. In the policy iteration, modifications are made to the selection of $H_i$, primarily for the following reasons: 

\begin{enumerate}
\item During the update of $u_h$ in Step 3, if the discrete contact
  set $\mathcal{C}_h$ has been updated, equation
    \eqref{eq:policy-update} can be seen as the discrete fractional
    Laplacian on the non-contact points $\mathcal{N}_h^0\setminus
    \mathcal{C}_h$. Instead, the values at the contact points
    $\mathcal{C}_h$ can be viewed as the Dirichlet ``boundary'' data.

\item In the discretization of the fractional Laplacian, the singular
  part is approximated by a scaled Laplace operator
    \cite{han2022monotone}. However, it is important to note that the
    regularity estimate for the obstacle problem, Proposition \ref{pp:interior-Holder},
    indicates that the true solution has only $C^{1,s}$ continuity
    over the boundary of the contact set. This limitation restricts
    the accuracy of such approximation.
\end{enumerate}

Based on the above, when discretizing the fractional Laplacian on
$\mathcal{N}_h^0\setminus \mathcal{C}_h$, it is necessary to restrict
the size of $H_i$ to ensure that $\Omega_i$ does not intersect with
the region relating to the discrete contact set $\mathcal{C}_h$.
Therefore, the following improved strategy is proposed:
\begin{equation} \label{eq:Hi-improved}
H_i = \min\{h_i^{\alpha_i}\delta_i^{1-\alpha_i},
  \theta\text{dist}(x_i, \mathcal{C}_h)\}  \quad \forall x_i \in
  \mathcal{N}_h^0,
\end{equation}
where $\theta$ is a constant that depends on the shape regularity of
the mesh. In the numerical experiments, $\theta$ is set to $1/4$.

Since the discrete contact set in policy iteration is continuously
updated, the improvement in the algorithm primarily lies in updating
the discretization of the fractional Laplacian through
\eqref{eq:Hi-improved} after updating the contact set. The improved
algorithm is outlined as follows.

\begin{algorithm}[htbp!] 
    \caption{Improved policy iteration for the fractional obstacle problem \eqref{eq:Gh}}\label{al:policy2}
    \KwIn{
    right hand side $f$, obstacle function $\psi$.}
    \KwOut{solution $u_h$.}
    \BlankLine
    Initialize $u_h^{(0)}(x_i) = \psi(x_i) \quad \forall x_i \in \mathcal{N}_h^0$;\\
    Initialize discrete fractional Laplacian: $(-\Delta)_h^{s,(0)} = (-\Delta)^s_h$.\\
    Iterate for $k \geq 0$:
    \begin{enumerate}
        \item Update discrete contact set $\mathcal{C}^{(k+1)}_h$:\\
        $\mathcal{C}^{(k+1)}_h:=\{x_i: (-\Delta)_h^{s,(0)}[u_h^{(k)}](x_i) - f(x_i) \geq u_h^{(k)}(x_i) -\psi(x_i) \}$.
         \item Update discrete fractional Laplacian:\\
        $(-\Delta)_h^{s,(k+1)}$ by choosing $H_i = \min\{h_i^{\alpha_i}\delta_i^{1-\alpha_i}, \theta\text{dist}(x_i, \mathcal{C}_h^{(k+1)})\}$.
        \item Update solution $u_h^{(k+1)}$ such that 
        \begin{equation} \label{eq:policy-update2}
        \begin{aligned}
            u_h^{(k+1)}(x_i) -\psi(x_i) &= 0\quad \forall x_i \in \mathcal{C}^{(k+1)}_h, \\
            (-\Delta)^{s,(k+1)}_h[u_h^{(k+1)}](x_i) -f(x_i) &=0 \quad \forall x_i \in \mathcal{N}_h^0\setminus \mathcal{C}^{(k+1)}_h.
        \end{aligned}
        \end{equation}
        \item If $u_h^{(k+1)}=u_h^{(k)}$, then stop; Otherwise, continue to Step 1.
    \end{enumerate}
\end{algorithm}

\begin{Remark}[convergence profile of Algorithm \ref{al:policy2}] \label{rk:policy2-convergence}
In the improved algorithm, it is important to note that the solutions
  in the consecutive steps correspond to different discretized
  fractional Laplacian operators, as seen in
  \eqref{eq:policy-update2}. Therefore, it is not possible to directly
  apply Theorem \ref{tm:policy-convergence} to provide a rigorous
  convergence analysis. However, it is worth noting that although the
  discretized fractional Laplacian operators vary across steps, they
  are all appropriate discretizations of the continuous operator and
  maintain monotonicity. From this perspective, the convergence
  behavior of the improved policy iteration algorithm should be
  similar to the original version. Indeed, this observation has been
  confirmed in numerical experiments, where the improved algorithm
  exhibits better performance, particularly near the contact
  interface.
\end{Remark}

\section{Numerical Experiments}  \label{sec:numerical}
In this section, several numerical experiments are conducted to test the stability and effectiveness of the methods, as well as the convergence of the improved policy iteration.

According to the discussion in Section \ref{subsec:regularity}, the fractional obstacle problem exhibits low regularity at the domain boundary, which is consistent with the fractional linear problem. To address this issue, various approaches have been proposed, including the use of graded grids. In this context, as discussed in \cite{acosta2017fractional, borthagaray2019weighted, han2022monotone}, the concept of {\it graded grids} is introduced with a parameter $h$. Let $\mu \geq 1$ be a constant such that for any $T\in \mathcal{T}_h$, 
\begin{equation}\label{eq:graded_grids}
    h_T \simeq
    \left\{
    \begin{aligned}
        &h^\mu\quad &\text{ if } T\cap \partial \Omega \neq \varnothing,\\
        &h\text{dist}(T,\partial \Omega)^{\frac{\mu-1}{\mu}} \quad &\text{ if } T \cap \partial \Omega = \varnothing.
    \end{aligned}
    \right.
\end{equation}  
It is worth noting that the quasi-uniform grid corresponds to the case where $\mu = 1$. Additionally, the choice of $\mu$ has an impact on both the grid quality and the overall computational complexity. The value of $\mu$ will be specified in each experiment.

In the discretization of the fractional Laplacian operator, the scale $H_i$ of the singular approximation part not only depends on the grid but also relies on $\alpha_i$, as shown in Equation \eqref{eq:Hi}. The optimal choice of $\alpha$ depends on the local smoothness of the solution or the smoothness of the forcing term $f$, as discussed in \cite{han2022monotone}. In the tests conducted in this section, a smooth $f$ is assumed, and based on the findings in \cite[Theorem 6.1]{han2022monotone}, the optimal choice of $\alpha$ is used, i.e., $\alpha = \frac12$.


\subsection{Convergence order test} \label{subsec:test1}
In this 1D test, the convergence rate is tested on both uniform and graded grids. The domain considered is $\Omega=(-1,1)$, and an explicit solution for problem \eqref{eq:fractional_obstacle} is constructed as follows,
$$
u(x) = \frac{2^{-2s}\Gamma(n/2)}{\Gamma(n/2+s)\Gamma(1+s)}(1-|x|^2)_+^s.
$$
A direct calculation shows that
$$(-\Delta)^s u=1\quad  x \in \Omega,\quad u=0\quad  x \in \Omega^c.$$
Next, the following forcing and obstacle terms are considered:
$$
f(x)=1-5(\frac{1}{2}-|x|)_+ \quad\text{and} \quad \psi = u(x)-\frac{1}{2}(x^2-\frac{1}{4})_+,
$$
so that $(-\Delta)^s u-f>0$ in $B_\frac{1}{2}(0)$ and $(-\Delta)^s u-f=0$ outside of this set. Consequently, $u$ represents the solution of problem \eqref{eq:fractional_obstacle}, and the contact set is $B_{\frac{1}{2}}(0)$.

\begin{figure}[h]
	\centering
	\includegraphics[scale=0.4]{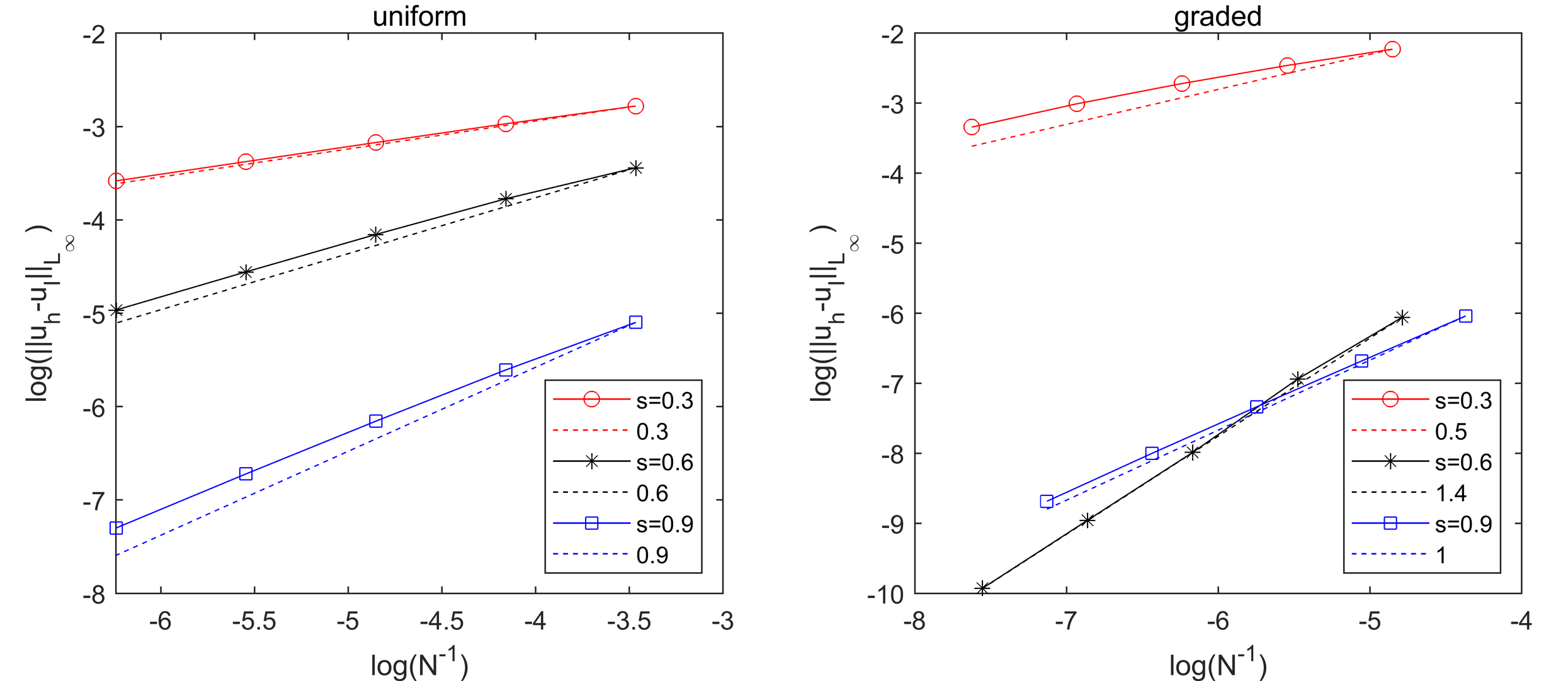}
	\caption{Experiment 1: Observed convergence rates for the discrete solutions of the fractional obstacle problems with $s=0.3, 0.6, 0.9$, computed over uniform and graded grids with $\mu = \frac{2-s}{s}$.}
 \label{figure:1derror}
\end{figure}

Figure \ref{figure:1derror} illustrates the convergence rates for both uniform and graded grids. According to \cite[Section 7.1]{han2022monotone}, the optimal choice for the pointwise estimate of the linear fractional Laplacian is $\mu = \frac{2-s}{s}$, which will also be used in this test. The convergence rate over uniform grids is observed to be approximately $s$, which is consistent with the behavior observed in the quasi-uniform grids for the linear case. Moreover, the graded grid yields a significantly improved convergence rate. It is worth noting that for the linear case, the convergence order under the optimal choice is $2 - s$, which is also observed for some cases (e.g., $s = 0.6$). However, for other cases, the strategy of choosing $H_i$ in \eqref{eq:Hi-improved} based on the local distance to the contact set may lead to different convergence rates compared to those analyzed for the linear case.

\subsection{Quantitative behavior and comparison of different policy iterations}
In this section, Problem \eqref{eq:fractional_obstacle} is investigated over the domain $\Omega=(-1,1)$. The force is set as $f=0$, and the obstacle function is given by
$$
\psi = 1-4|x-1/4|.
$$

\begin{figure}[h]
	\centering
	\includegraphics[scale=0.3]{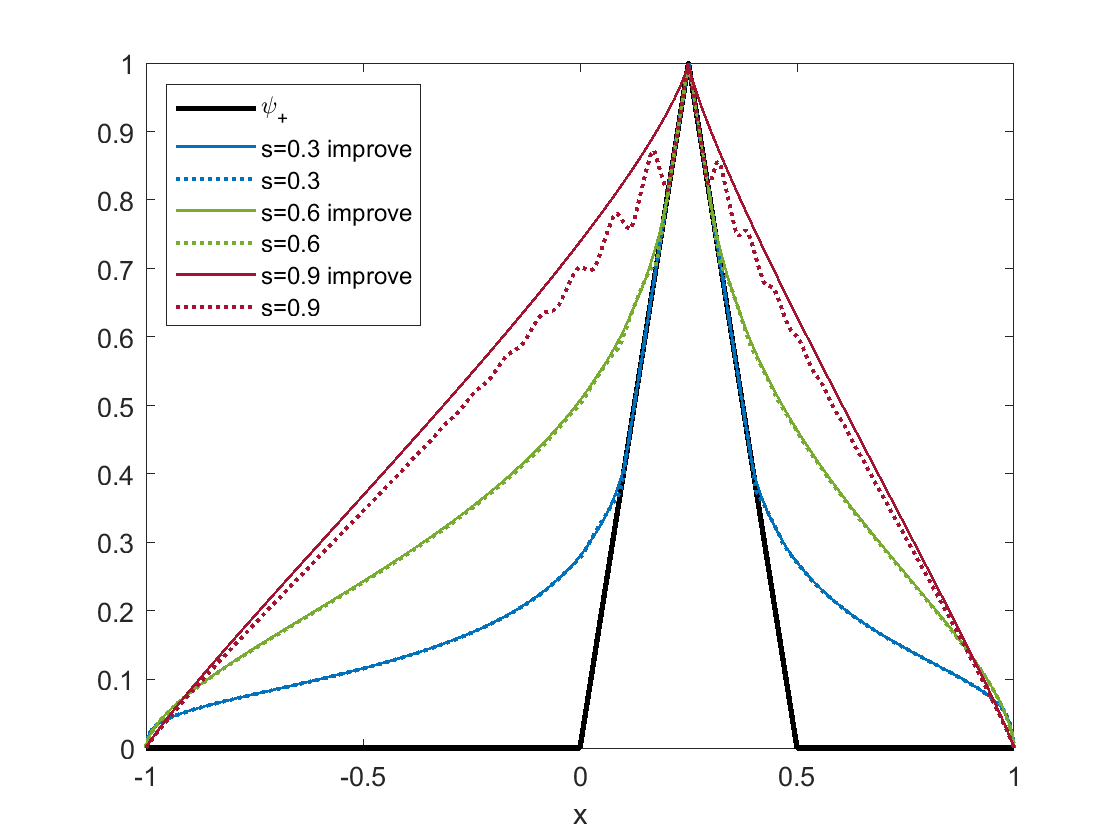}
	\caption{Numerical solutions to the fractional obstacle problem for $s = 0.3, 0.6, 0.9$, computed using the standard policy iteration (Algorithm \ref{al:policy}) and its improved version (Algorithm \ref{al:policy2}), respectively.}
 \label{figure:1duh}
\end{figure}

In Figure \ref{figure:1duh}, the numerical solutions $u_h$ obtained for different values of $s$ are presented below. A clear qualitative difference is observed between the solutions for different choices of $s$. When $s=0.9$, the discrete solution resembles the expected solution of the classical obstacle problem, where the operator $(-\Delta)^s$ is effectively replaced by $-\Delta$. As $s$ decreases, the solution approaches $\psi_+$. Throughout the experiments, a consistent observation is made that as $s$ increases, the contact set decreases and always contains the point $x_0$.

The difference between the improved method and the original method is remarkable. Observations reveal that when using the original method with larger values of $s$, oscillations occur near the free boundary due to the handling of the singular part across it.On the other hand, the improved method yields a numerically smoother solution without oscillations. As $s$ decreases, the solutions obtained by the two methods become closer to each other, and the differences diminish.


\subsection{Convergence history of improved policy iteration}
Next, the convergence history of the improved policy iteration (Algorithm \ref{al:policy2}) is examined. The experiment is conducted with $f=1$, and the obstacle function is defined as follows:
$$
\psi(x)= 3-6|x-1/4|.
$$

\begin{figure}[h]
	\centering
	\includegraphics[scale=0.66]{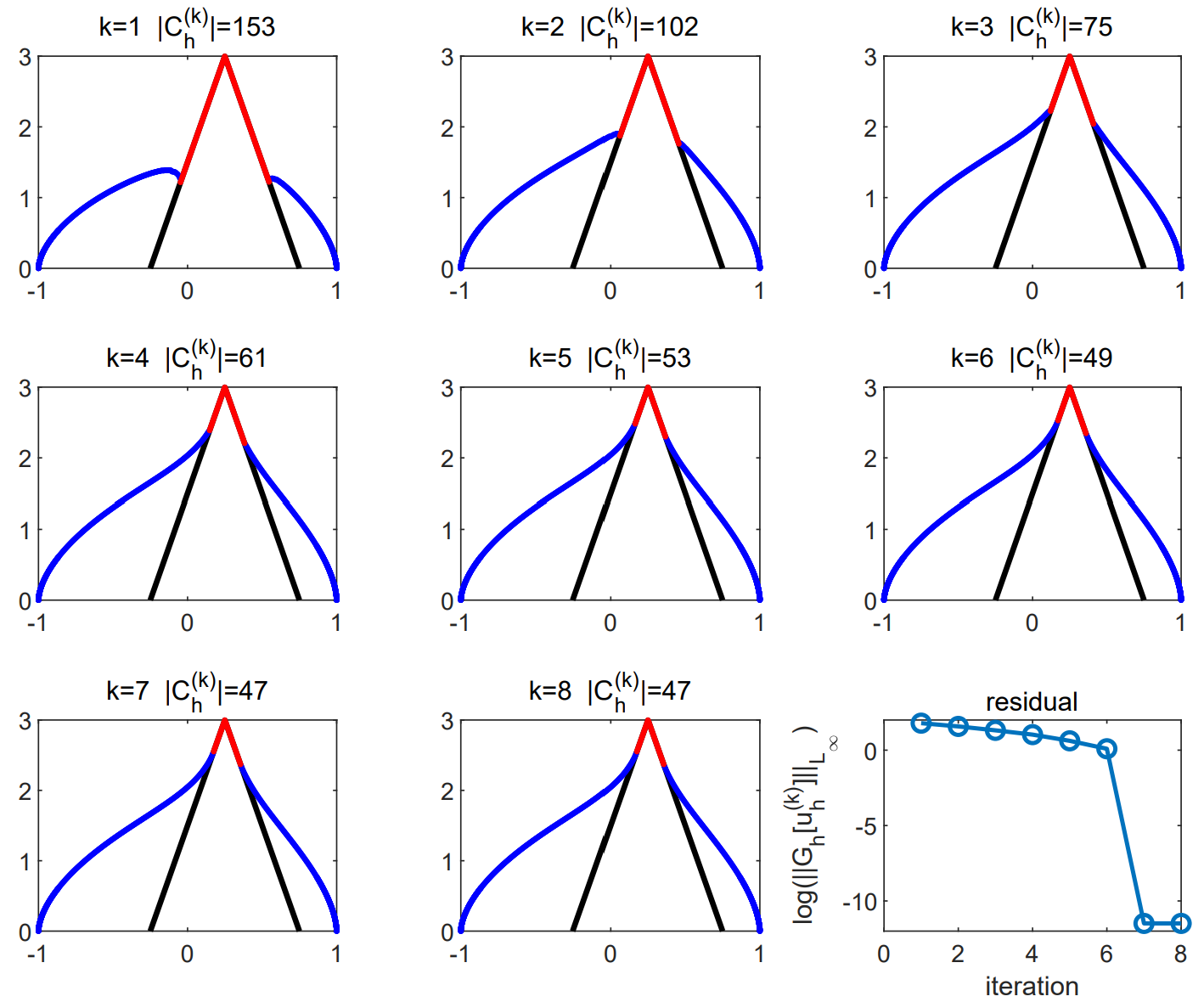}
	\caption{Convergence history of the improved policy iteration (Algorithm \ref{al:policy2}) with $s=0.6$ and 953 free vertices, including residual decay versus iterations (last picture). The red points indicate the contact set, and the blue points represent the non-contact set.}
 \label{fg:policy}
\end{figure}

As mentioned in Remark \ref{rk:policy2-convergence}, it is highly likely that the improved algorithm possesses the property of an increasing solution sequence (or decreasing discrete contact set) as stated in Theorem \ref{tm:policy-convergence} for the standard algorithm. In this experiment, as demonstrated in Figure \ref{fg:policy}, the convergence profile is observed through the values of $|\mathcal{C}_h^{(k)}|$. 
\begin{table} [h]
	\centering	
	\begin{tabular}{|c|c c c c|c c c c|c c c c| }  
		\hline
         & \multicolumn{4}{|c|}{$s=0.3$ }&\multicolumn{4}{|c|}{$s=0.6$}&\multicolumn{4}{|c|}{ $s=0.9$ }\\ \hline
		$N$ &126&254& 510&1022 &118&237 & 476&953 & 77& 155&311& 623\\ \hline  
	iterations&4&5&5 &6 & 5&6 & 7& 8& 6& 7&9&11  \\ \hline  
	\end{tabular}  
 \caption{Number of iterations for $s=0.3, 0.6, 0.9$ using graded grids ($\mu = \frac{2-s}{s}$), with different number of degrees of freedom (DOFs).}  
 \label{tb:policy}
\end{table}

Furthermore, the actual number of iterations obtained in the experiments is presented in Table \ref{tb:policy}. It can be observed that the actual number of iterations required is significantly smaller than the theoretical upper bound $N$. Additionally, as the grid is refined, the algorithm's performance remains stable. This demonstrates the efficiency of our algorithm.

\subsection{2D test}
Finally, problem \eqref{eq:fractional_obstacle} is considered in the domain $B_1(0)\subset \mathbb{R}^2$ with $f=0$, and the obstacle function is given by
$$
\psi = \frac{1}{2}-|x-x_0|,\quad  where\  x_0=(\frac{1}{4},\frac{1}{4}).
$$
\begin{figure}[h]
	\begin{minipage}[t]{0.45\linewidth}
		\centering
		\includegraphics[width=5cm,height=3.5cm]{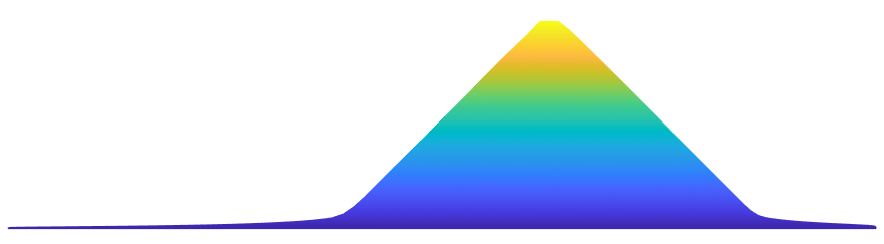}
	\end{minipage}
	\begin{minipage}[t]{0.45\linewidth} 
		\hspace{5mm}
		\includegraphics[width=5cm,height=3.5cm]{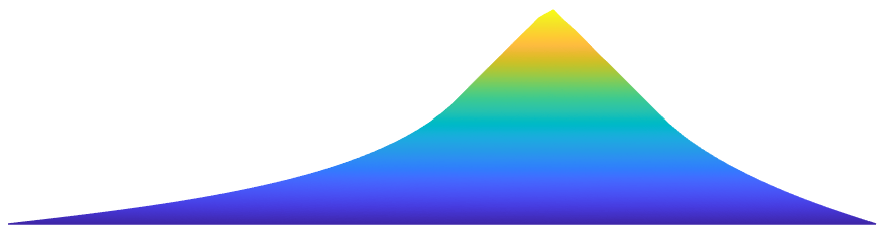}
	\end{minipage}
 
	\begin{minipage}[t]{0.45\linewidth}
		\centering
		\includegraphics[scale =0.5]{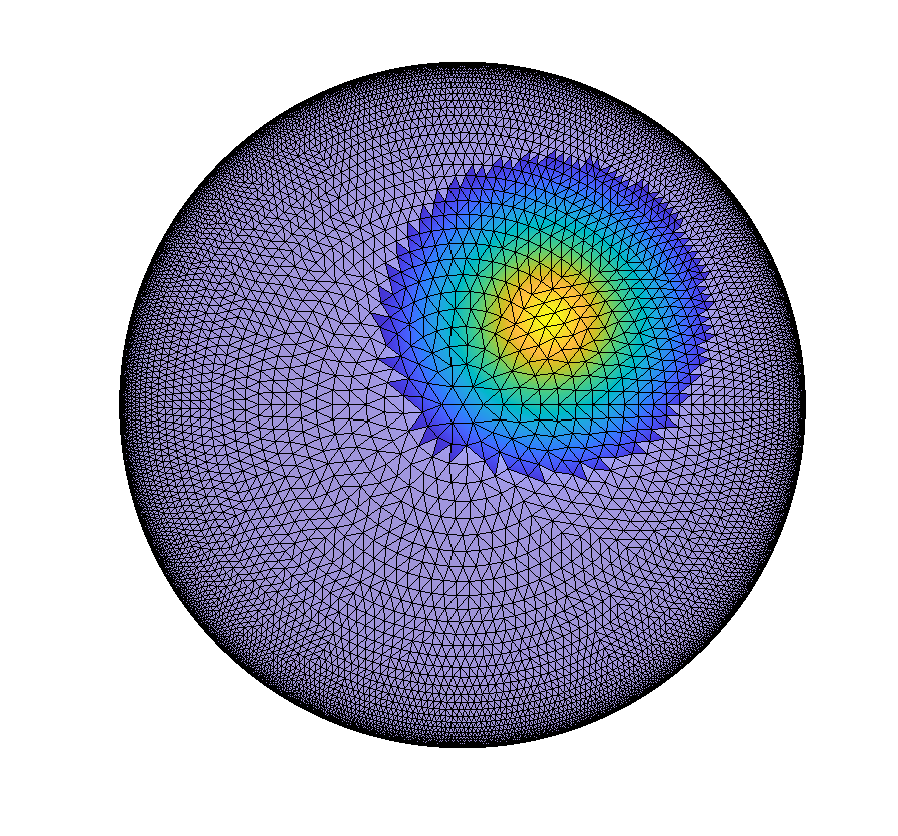}
	\end{minipage}
	\begin{minipage}[t]{0.45\linewidth} 
		\hspace{3mm}
		\includegraphics[scale =0.5]{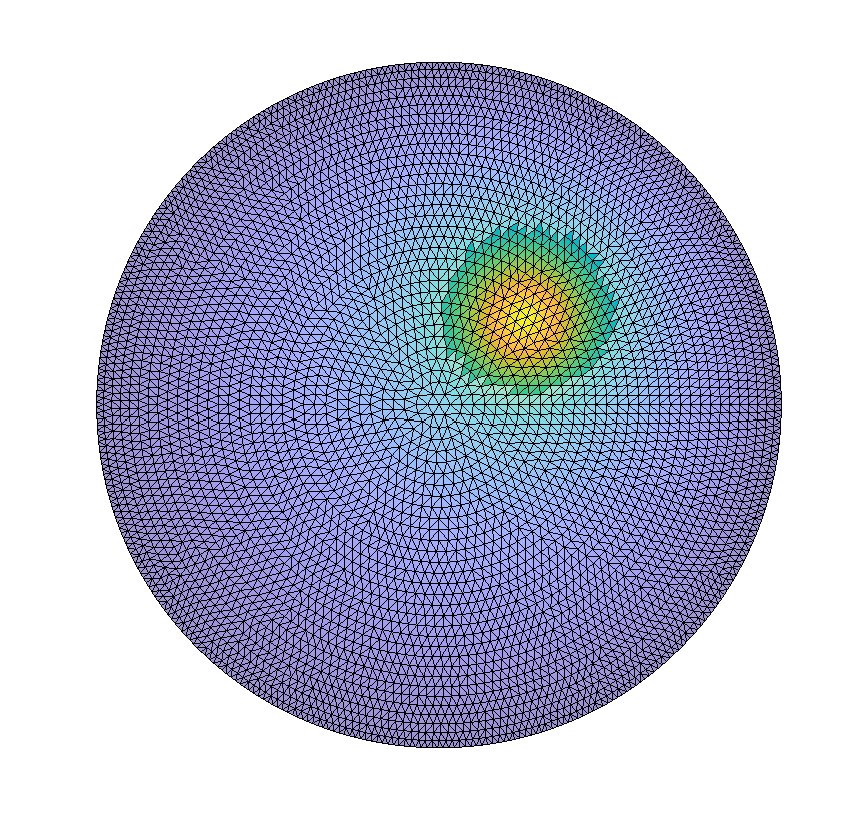}
	\end{minipage}
 \caption{Discrete solutions to the fractional obstacle problem for $s = 0.1$ (left) and $s = 0.9$ (right). Top: lateral view. Bottom: top view, with the discrete contact set highlighted.}
\end{figure}

For $s=0.1$, the numerical solution requires 13,395 degrees of freedom, and the number of improved policy iterations is 4. Conversely, for $s=0.9$, the numerical solution involves 4,253 degrees of freedom, and the number of improved policy iterations is 10. The number of iterations increases as $s$ grows, but significantly less than $N$, as also observed in the 1D case. As expected, the 2D numerical results demonstrate similar qualitative characteristics as observed in 1D.

\section{Conclusion}
\commenttwo{In this paper, a discrete scheme for the fractional obstacle problem is proposed, which is based on the monotone discretization for the fractional Laplace operator. 
Utilizing the distinctive structure of the problem \eqref{eq:Gh} and conducting a comprehensive study of the discrete operator, the study reveals that the nonlinear discrete operator $\mathcal{G}_h$ upholds the discrete comparison principle.
Based on this property, the existence, uniqueness, and uniform boundedness of numerical solutions are established.
}

\commenttwo{Due to the limited regularity of the true solution near the domain boundary, a graded grid has been introduced to capture this behavior. 
For solving nonlinear problems, the policy iteration method is used. Benefiting from the monotonicity of the discrete scheme of the fractional Laplace operator, the policy iterative process can converge in finite iterations.
Moreover, considering the reduced regularity of the actual solution near the contact set boundary, the discretization of the fractional Laplacian has been adaptively refined by iteratively updating contact nodes. This refinement has led to an improved policy iteration approach. 
In contrast to the conventional policy iteration, this improved method demonstrates superior numerical performance across a range of numerical experiments. One-dimensional and two-dimensional examples are provided to support the theoretical results.
}

\appendix
\section[\appendixname~\thesection]{proof of Proposition \ref{pp:boundary-Holder} under a weak assumption of $\psi$} \label{ap:holder}

Proposition \ref{pp:boundary-Holder} is proven, assuming $\psi \in L^\infty(\Omega)$. The proof follows closely the global H\"older estimate for the linear problem \cite{ros2014dirichlet}, so only the main thread will be sketched. First, some elementary tools are recalled.

\begin{Proposition}[Corollary 2.5 in \cite{ros2014dirichlet}]\label{pp:semi_holder_esti}
	Assume that $w\in C^\infty(\mathbb{R}^n)$ is a solution of $(-\Delta)^s w = h$ in $B_2$. Then for every $\beta \in(0,2s)$,
	\[
	\begin{aligned}
		\|w\|_{C^\beta({B_{1/2}})}\leq C \left(\|(1+|x|)^{-n-2s}w(x)\|_{L^1(\mathbb{R}^n)} + \|w\|_{L^\infty(B_2)} + \|h\|_{L^\infty(B_2)}\right)
	\end{aligned}
	\]
	where the constant $C$ depends only on $n, s$ and $\beta$.
\end{Proposition}

\begin{Lemma}[$L^\infty$ estimate of obstacle problem, \cite{musina2017variational}]\label{lm:Linfty_estimate}
 For $\psi \in L^\infty(\Omega)$ and $f = 0$, the solution $u$ of (\ref{eq:fractional_obstacle}) satisfies $u \in L^\infty(\Omega)$ with the following bounds:
	\[
	\max\{\psi,0\} \leq u \leq \|\max \{\psi,0\}\|_{L^\infty(\Omega)}.
	\]
\end{Lemma}

The proof of Lemma \ref{lm:Linfty_estimate} follows from the comparison principle. As a direct consequence, the $L^\infty$ estimate of the fractional Laplace equation implies the following inequality for $f\in L^\infty(\Omega)$ and $\psi \in L^\infty(\Omega)$:
\[
\|u\|_{L^\infty(\Omega)}\leq C\left(\|f\|_{L^\infty(\Omega)} + \|\psi\|_{L^\infty(\Omega)}\right).
\]

\begin{Lemma}[supersolution, Lemma 2.6 in \cite{ros2014dirichlet}]\label{lm:supersolution}
	There exist $C_1>0$ and a radial continuous function $\varphi_1$ satisfying 
	\[
	\left\{
	\begin{aligned}
		&(-\Delta)^s \varphi_1 \geq 1 \quad &&x\in B_4\setminus B_1,\\
		&\varphi_1\equiv 0\quad &&x\in B_1,\\
		&0\leq \varphi_1 \leq C_1(|x|-1)^s\quad&& x\in B_4\setminus B_1,\\
		&1\leq \varphi_1 \leq C \quad&& x\in \mathbb{R}^n \setminus B_4.
	\end{aligned}
	\right.
	\]
\end{Lemma}

By the Lemma \ref{lm:Linfty_estimate} and Lemma \ref{lm:supersolution}, one could construct an upper barrier for $|u|$ by scaling and translating the supersolution.
The proof of Lemma \ref{lm:dist_dependence} is similar to the Lemma 2.7 in \cite{ros2014dirichlet}. The only difference is for each point $x_0 \in \partial \Omega$, when constructing the upper barrier on the ball touched $x_0$ from outside, the radius of the ball not only depends on the exterior ball condition of $\Omega$, but also depends on $r_0$ which appeared in the assumption \eqref{eq:ass-contact-dist} for the contact set.

\begin{Lemma}[$\delta(x)^s$-boundary behavior]\label{lm:dist_dependence}
	Let $\Omega$ be a bounded satisfying the exterior ball condition and let $f\in L^\infty(\Omega), \psi \in L^\infty(\Omega)$, u is the solution of (\ref{eq:fractional_obstacle}).
	Then 
	\[
	|u(x)|\leq C(\|f\|_{L^\infty(\Omega)}+\|\psi\|_{L^\infty(\Omega)})\delta^s(x),
	\]
	where $C$ is a constant depending only on $\Omega$, $s$ and $r_0$. Here, $\delta(x) := \mathrm{dist}(x, \partial \Omega)$. 
\end{Lemma}


The following H\"older seminorm estimate is provided by the above Lemma, which plays a vital role in the proof of Proposition \ref{pp:boundary-Holder}.

\begin{Proposition}[improved interior H\"older estimate] \label{lm:improved-interior-Holder}
	Let $\Omega$ satisfies the exterior ball condition, $f,\psi \in L^\infty(\Omega)$, u is the solution of (\ref{eq:fractional_obstacle}), then $u\in C^\beta(\Omega\setminus\Lambda), \;\beta \in (0,2s)$ and for all $x_0\in \Omega\setminus \Lambda$ wa have the following seminorm estimate in $B_R(x_0) := B_{\delta^*(x_0)/2}(x_0)$ 
	\[
	{|u|}_{C^\beta({B_R(x_0)})}\leq C(\|\psi\|_{L^\infty(\Omega)} + \|f\|_{L^\infty(\Omega)}) R^{s-\beta},
	\]
	where $\delta^*(x):=\min\{\delta(x),\text{dist}(x,\Lambda)\}$ and the constant $C$ depending only on $\Omega, s$, $\beta$ and $r_0$.
\end{Proposition}
\begin{proof}
Using the standard mollifier technique, it can be assumed that $u$ is smooth. Note that $B_R(x_0)\subset B_{2R}(x_0) \subset \Omega.$ Let $\tilde{u}(y) := u(x_0 + Ry)$. The equation for $\tilde{u}$ is given by:
	\begin{equation}\label{eq:cond1}
		\begin{aligned}
			(-\Delta)^s \tilde{u}(y) = R^{2s} f(x_0+Ry)\quad \forall y \in B_1.
		\end{aligned}
	\end{equation}
 Furthermore, by employing the inequality $|u(x)|\leq C\left(|u|{L^\infty(\mathbb{R}^n)} + |f|{L^\infty(\Omega)}\right)\delta^s(x)$ in $\Omega$ (see Lemma \ref{lm:dist_dependence}), the following result can be deduced:
	\begin{equation}\label{eq:cond2}
		\|\tilde{u}\|_{L^\infty(B_1)}\leq C \left(\|u\|_{L^\infty(\mathbb{R}^n)} + \|f\|_{L^\infty(\Omega)}  \right)R^s.
	\end{equation}
	Further, \eqref{eq:cond2} and Lemma \ref{lm:dist_dependence} also imply that 
	$$
	|\tilde{u}(y)|\leq C\left(\|u\|_{L^\infty(\mathbb{R}^n)} + \|f\|_{L^\infty(\Omega)}+\|\psi\|_{L^\infty(\Omega)}\right)R^s(1+|y|^s)\quad \forall y \in \mathbb{R}^n,
	$$ 
	whence
	\begin{equation}\label{eq:cond3}
		\|(1+|y|)^{-n-2s} \tilde{u}(y)\|_{L^1(\mathbb{R}^n)} \leq C\left(\|u\|_{L^\infty(\Omega)}+\|f\|_{L^\infty(\Omega)}\right)R^s.
	\end{equation}
	
	Now, one can use the Proposition \ref{pp:semi_holder_esti}, which the $C^{\beta}$ seminorm of $\tilde{u}$ can be bounded by \eqref{eq:cond1}, \eqref{eq:cond2} and \eqref{eq:cond3}, and obtain 
	\[
	\|\tilde{u}\|_{C^\beta({B_{1/4}})}\leq C(\|\psi\|_{L^\infty(\Omega)}+\|f\|_{L^\infty(\Omega)})R^s
	\]
	for all $\beta \in(0,2s)$, where $C = C(\Omega, s,\beta,r_0)$ and Lemma \ref{lm:Linfty_estimate} is used to bound $\|u\|_{L^\infty(\Omega)}$.
 \commenttwo{
Finally, the relationship
	\[
	|u|_{C^\beta({B_{R/4}(x_0)})} = R^{-\beta}{|\tilde{u}|}_{C^\beta({B_{1/4}})},
	\]	
implies that
 \[
 |u|_{C^\beta({B_{R/4}(x_0)})}\leq C(\|\psi\|_{L^\infty(\Omega)}+\|f\|_{L^\infty(\Omega)})R^{s-\beta}.
 \]
 Hence, by a standard covering argument, it can be found the $C^\beta$ seminorm of $u$ in ${B_R(x_0)}$.
 }
\end{proof}

Having Proposition \ref{lm:improved-interior-Holder} (improved interior H\"older estimate), the desired Proposition \ref{pp:boundary-Holder} can be proved using the same technique that appeared in the proof of Proposition 1.1 in \cite{ros2014dirichlet} provided $\beta = s$.

\bibliography{frac}
\bibliographystyle{plain}

\end{document}